\documentclass[11pt]{amsart} 
\usepackage{amssymb,amsmath,latexsym,enumerate,graphicx,bbm,mathptmx,microtype,cite,ifthen,color}
\allowdisplaybreaks
\usepackage{cmbright}

\newtheorem{theorem}{Theorem} % [section] 
\newtheorem*{lrc}{Lonely Runner Conjecture} 
\newtheorem*{slrc}{Shifted Lonely Runner Conjecture} 
\newtheorem*{dat}{Dirichlet's Approximation Theorem} 
\newtheorem{corollary}[theorem]{Corollary}
\newtheorem{proposition}[theorem]{Proposition}
\newtheorem{lemma}[theorem]{Lemma}

\newtheorem{exam}{Example}

\newtheorem*{rem}{Remarks}

\numberwithin{equation}{section}

\renewcommand\th{^{\text{th}}}

\newcommand\commentout[1]{}
\newcommand\Def[1]{{\bf #1}}

\newcommand\onenorm[1]{\left|\left|{#1}\right|\right|_1}
\newcommand\zdist[1]{\left|\left|{#1}\right|\right|_\ZZ}
\newcommand{\norm}[1]{\lVert #1\rVert}

\newcommand\prob{\operatorname{Prob}}

\newcommand\lon{\operatorname{lon}} 
\newcommand\width{\operatorname{width}} 
\newcommand\flt{\operatorname{flt}}

\newcommand\ZZ{\mathbb{Z}}

\newcommand\RR{\mathbb{R}}

\newcommand\cZ{\mathsf{Z}}

\newcommand\ba{\mathbf{a}}
\newcommand\bb{\mathbf{b}}
\newcommand\bc{\mathbf{c}}
\newcommand\be{\mathbf{e}}
\newcommand\bm{\mathbf{m}}
\newcommand\bn{\mathbf{n}}
\newcommand\bs{\mathbf{s}}
\newcommand\bu{\mathbf{u}}

\newcommand\bx{\mathbf{x}}
\newcommand\by{\mathbf{y}}
\newcommand\bz{\mathbf{z}}
\newcommand\bzero{\mathbf{0}}
\newcommand\bone{\mathbf{1}}

\makeatletter % to repeat theorem numbers
\newtheorem*{rep@theorem}{\rep@title}\newcommand{\newreptheorem}[2]{%
\newenvironment{rep#1}[1]{%
\def\rep@title{\bf #2 \ref{##1}}%
\begin{rep@theorem}}%
{\end{rep@theorem}}}
\makeatother
\newreptheorem{theorem}{Theorem}

\begin{document}

\title{Lonely Runner Relations}

\author{Matthias Beck}
\address{Department of Mathematics\\
         San Francisco State University\\
         San Francisco, CA 94132\\
         U.S.A.}
\email{mattbeck@sfsu.edu}
\urladdr{https://matthbeck.github.io}

\author{Samuel Everett}
\address{Department of Computer Science\\
         University of Chicago\\
         Chicago, IL 60637\\
         U.S.A.}
\email{same@uchicago.edu}
\urladdr{https://www.samueleverett.com}

\dedicatory{Dedicated in honor of 150 combined years of the lonely runner J\"org Wills.}

\begin{abstract}
We study the \emph{Lonely Runner Conjecture} (LRC), conceived by J\"org M.~Wills in the 1960's: Given positive integers
$n_1, n_2, \dots, n_k$, there exists a positive real number $t$ such that for all $1 \le j \le k$ the distance of $t \,
n_j$ to the nearest integer is at least $\frac{ 1 }{ k+1 }$.  We prove that for any counterexample or tight instance
$\mathbf{n}$ of LRC, $\mathbf{m} \cdot \mathbf{n} = 0$ for some $\mathbf{m} \in \mathbb{Z}^k$ with $0 < \| \mathbf{m} \|_1 \le \min(2k+3, \ \frac{ k+1 }{ k-1 } \mathrm{flt}(k))$ where $\mathrm{flt}(k)$ denotes Khinchin's (1948)
\emph{flatness constant} limiting the lattice width of a $k$-dimensional convex body without interior integer points.
In other words, potential counterexamples to LRC lie on a finite set of hyperplanes in the parameter space.
Our proofs use Fourier analysis and a geometric reformulation of LRC, and our results generalize to the situation of
shifted lonely runners of varying measures of loneliness. Our results imply and generalize a theorem of Czerwi\'nski
(2012) that when we choose $\mathbf{n}$ at random then, with probability tending to 1, the measure of loneliness $\frac{ 1
}{ k+1 }$ can be replaced by $\frac 1 2 - \epsilon$.
We also show that the lattice width of the polyhedra appearing in conjunction with LRC grows linearly with the
dimension.
\end{abstract}


% 11J71 Distribution modulo one
% 42A16 Fourier coefficients, Fourier series of functions with special properties, special Fourier series
% 52B12 Special polytopes (linear programming, centrally symmetric, etc.)
% 52C07 Lattices and convex bodies in $n$ dimensions

\thanks{We thank Gennadiy Averkov, Benjamin Bedert, Ansgar Freyer, Serkan Ho\c{s}ten, Noah Kravitz, Dasha Poliakova, Paco Santos, and Matthias Schymura for helpful discussions and pointers to the literature.
\\ \indent
AI tools were used for the probability estimates in our proof of Theorem~\ref{thm:maxwidthomega}. All other proofs and
all the writing were done with human intelligence.
}

\date{7 November 2026}

\maketitle

% --------------------

\section{Introduction}

The theme of our paper is given by the following conjecture, raised by J\"org M.~Wills in the
1960's~\cite{willslonelyrunnerstart,willslonelyrunner}.

\begin{lrc}
Given distinct positive integers $n_1, \dots, n_k$, there exists a real number~$t$ such that for all $1 \le j \le k$, the distance of $t n_j$ to the nearest integer is at least $\frac{ 1 }{ k+1 }$.
\end{lrc}

Wills originally formulated this conjecture for \emph{real} numbers $n_1, \dots, n_k$, but it can be reduced to the rational and thus integral case~\cite{bohmanholzmankleitman,henzemalikiosis}.
The lower bound $\frac{ 1 }{ k+1 }$ is best possible, as the case $n_j = j$ for $1 \le j \le k$ and the following
classic result (see, e.g., \cite{cassels}) shows.

\begin{dat}
For every real number $t$ and positive integer $k$, there exists $q \in \{ 1, \dots, k \}$ such that the distance of
$tq$ to the nearest integer is at most $\frac{ 1 }{ k+1 }$. 
\end{dat}

Wills asked if this theorem can be improved by replacing $\{ 1, \dots, k \}$ with a different set of $k$ numbers, and
his conjecture says it cannot.
% Will's email from 16 April 2026: Of course i admired his work, in particular his famous theorem on Uniform distribution (Gleichverteilung). So, in the sixties, the idea came to me, to ask For the opposite, the worst-case Dioph. Approx. Of irrationals. That was the beginning…
The name \emph{Lonely Runner Conjecture} (LRC), introduced by Goddyn in~\cite{bieniagoddynetal}, stems from the charming
model of $k+1$ runners going at different constant speeds around a circular track of length~1 (having started at the
same place and time); the conjecture says that each of them will at some point have distance at least $\frac{ 1 }{
k+1 }$ to the other runners (and the vector $\bn$ records the differences of the velocities relative to the given runner).
% When Wills saw the lonely runner paper, he emailed one of the authors, Luis Goddyn, to congratulate him on “this wonderful and poetic name.” Goddyn’s reply: “Oh, you are still alive.”
LRC enjoys connections to various fields, e.g., number theory, harmonic analysis, graph theory, and discrete
geometry and has been proved for $k \le 13$ via the following philosophy.
Tao~\cite{taolonelyrunner} showed the first \emph{finite-checking result} for LRC, i.e., to prove LRC in any given
dimension, it suffices to verify it for finitely many $\bn$, with an
explicit bound on the velocities. This bound has subsequently been
improved~\cite{malikiosisSantosSchymura,giriKravitz}, which together with some nontrivial number
theory~\cite{rosenfeld}, was good enough to verify the current records~\cite{allikvere,sungkawichaiTrakulthongchai}.

Our first main result further limits the set of counterexamples via certain relations among the involved velocities.
To state it, we first introduce some terminology.
Let $\zdist{ \mbox{\ \ } } $ denote distance to the nearest integer and, for $\bn \in \ZZ^k_{>0}$ let
\[
  \lon(\bn) \, := \max_{t \in [0,1]} \min_{1 \le j \le k} \zdist{t n_j} \, , 
\]
and so LRC asserts $\lon(\bn) \ge \frac{1}{k+1}$ for all $\bn$ with distinct entries. 
Wills proved already in~\cite{willslonelyrunner} that $\lon(\bn) > \frac{ 1 }{ 2k }$, 
% $\ge \frac{ 1 }{ 2k }$ because $\{ t \in [0,1] : \, \zdist{ t n_j } < \frac{ 1 }{ 2k } \}$ is a union of open arcs of length $\frac{ 1 }{ k n_j }$ on $[0,1]$, viewed as a torus, and so these sets cannot cover $[0,1]$
and the current best lower bound is $\frac{ 1 }{ 2k } + \frac{ 1 }{ k^{ \frac 5 3 + o(1) } }$~\cite{bedert}.

We call $\bn$ a \Def{counterexample} if $\lon(\bn) < \frac{1}{k+1}$ and a \Def{tight instance} if $\lon(\bn) = \frac{1}{k+1}$.
Tight instances were studied, e.g., in~\cite{goddynwong}.

The \Def{lattice width} of a convex body $C \subset \RR^k$ is
\[
  \width(C) \, := \min_{ \bm \in \ZZ^k \setminus \bzero } \ \max_{ \ba, \bb \in C } \bm \cdot (\ba - \bb)
\]
and $C$ is \Def{hollow} if $C$ contains no integer points in its interior.
Khinchin's famous \emph{flatness theorem} \cite{khinchin} asserts the existence of a (minimal) number $\flt(k)$
such that any hollow convex body $C \subset \RR^k$ has $\width(C) \le \flt(k)$.
This quantity is the \Def{flatness constant} in dimension~$k$.
It has been extensively studied in convex geometry and integer optimization; we know, e.g., that $\flt(k) \in O(k
\log(2k)^3)$~\cite{reisrothvoss}, and that $k$-dimensional centrally symmetric hollow bodies have a maximum lattice width in $O(k \log k)$~\cite{banaszczyk}. Just as much seems to be open for exploration, e.g., we know very few $k$-dimensional polytopes with lattice width larger than~$k$~\cite{codenottisantos}.
We denote, as usual, $\onenorm \bm := |m_1| + \dots + |m_k|$.

\begin{theorem}\label{thm:mainlr}
Let $k \ge 2$ and let $\bn \in \ZZ^k_{>0}$ be a counterexample to or a tight instance of the Lonely Runner Conjecture, i.e., $\lon(\bn) \le \frac{1}{k+1}$.
Then 
\begin{enumerate}[{\sf (a)}]
  \item $\bm \cdot \bn = 0$ for some $\bm \in \ZZ^k$ with $0 < \onenorm{\bm} \le 2k+3$ and $m_1 + \dots + m_k$ odd, and
  \item $\bm \cdot \bn = 0$ for some $\bm \in \ZZ^k$ with $0 < \onenorm{\bm} \le \frac{k+1}{k-1} \flt(k)$.
\end{enumerate}
\end{theorem}

To illustrate this theorem within the parameter space, let
\[
  \diamond_k \, := \, \left\{ \bm \in \ZZ^k : \, 0 < \onenorm \bm \le \min \left( 2k+3, \ \frac{k+1}{k-1} \flt(k) \right) \right\} .
\]
Theorem~\ref{thm:mainlr} says that all potential counterexamples to, or tight instances of, the $k$-dimensional LRC 
live in the (finite) union of
hyperplanes $\bigcup \{ \bm \cdot \bx = 0 : \, \bm \in \diamond_k \}$ in $\RR^k$.
E.g., the tight instance $\bn = (1, \dots, k)$ featured in Dirichlet's theorem is a point in this union, as it lives on the hyperplane $2x_1 - x_2 = 0$.
Thus our results give a natural computational playground to look for counterexamples of LRC, as
well as new examples of tight lonely runner instances: check for possible such instances $\bn$ for which $\bm \cdot \bn = 0$ for some $\bm \in \ZZ^k$ with small 1-norm.\footnote{
Yes, absolutely: we tried, albeit only crudely so.
}
We remark that the qualitative description of Theorem~\ref{thm:mainlr}, though with arguably worse bounds on the size
of $\bm$, can with some work also be derived from \cite{giriKravitz} (combining Lemmas 3.3 and 4.4 and Theorem~4.1 in that paper).
Another version of Theorem~\ref{thm:mainlr} can also be derived from~\cite[Lemma~4.1]{bedert}, which implies that a
counterexample to LRC must satisfy some nontrivial relation with coefficients $0, \pm 1$. 
% Bedert's lemma says that if the set V of n speeds is dissociated, then it satisfies the LRC. So if V is a hypothetical counterexample to LRC, it cannot be dissociated, which means exactly that the speeds must satisfy some nontrivial relation with coefficients -1,0 and 1. So this would give a hyperplane given by some vector m with \ell^1 norm at most n.

We think of the two upper bounds (a) and (b) as complementary: the first gives a concrete formula, whereas the second
one depends on a quantity that is not known for a specific $k$, yet likely to be smaller than the first, and this
inequality holds in more general settings, as we will see below.
The two bounds stem from two quite different proof methods we employ in this paper. Each comes with
a (quite distinct) generalization of Theorem~\ref{thm:mainlr} and various consequences.

The first stems from Fourier analysis, carried out in Section~\ref{sec:fourier}.
The closest methodological precedent comes from Czerwi\'nski \cite{czerwinski}, who uses a convolution window and a product Fourier expansion whose surviving frequencies are additive relations among the velocities.
Tao also used Fourier estimates for multiple intersections to determine low-complexity additive relations from highly structured lonely runner
instances \cite{taolonelyrunner}, and Jensen employed Fourier series on the corresponding distance-threshold functions
\cite{jensen2026mixed}. Relatedly, Bedert used Riesz products and additive-dimension arguments to obtain substantially stronger loneliness bounds for dissociated velocity sets~\cite{bedert}.

The second bound in Theorem~\ref{thm:mainlr} stems from
the equivalent formulation for LRC in~\cite{lonelyrunnerpoly} via the existence of an integer point in a certain
$k$-dimensional polyhedron. It turns out that one can compute the lattice width of a zonotopal variant of this
polyhedron explicitly, and the flatness theorem then gives our result. 
We note that, consequently, we could replace $\flt(k)$ in the statement of Theorem~\ref{thm:mainlr}(b) by the (smaller)
flatness constant for symmetric polytopes or that for zonotopes.
This geometric setup and the resulting proofs are
in Section~\ref{sec:geometry}.
We remark that the idea of using lattice width or the flatness theorem in the context of lonely runner
polyhedron is not new and was employed in the afore-mentioned papers~\cite{henzemalikiosis,malikiosisSantosSchymura}. 

Our geometric results generalize (see Theorem~\ref{thm:maingeom} below) to the setting given by the following refined conjecture, also due to Wills.

\begin{slrc}
Given distinct positive integers $n_1, \dots, n_k$ and real numbers $s_1, \dots, s_k$, there exists a real number~$t$
such that for all $1 \le j \le k$, the distance of $t n_j + s_j$ to the nearest integer is at least $\frac{ 1 }{ k+1 }$.
\end{slrc}

This conjecture was recently disproved~\cite{blancocriadosantos}, with the smallest counterexample $\bn = (1, 2, 3,
4, 5)$; note again that this point satisfies the conclusion of Theorem~\ref{thm:mainlr}. 
An even more recent result~\cite{poliakova} quantifies the asymptotic failure of the shifted LRC (as $k$ grows).
Naturally, this situation gives rise to new questions, and we collect some of them in Section~\ref{sec:open}.

We will also see that Theorem~\ref{thm:maingeom} implies (and generalizes to the shifted lonely runner situation) 
a result of Czerwi\'nski~\cite{czerwinski} that when we choose
$\bn$ at random then, with probability tending to 1, the measure of loneliness $\frac{ 1 }{ k+1 }$ can be replaced by
$\frac 1 2 - \epsilon$; the detailed statement is in Corollary~\ref{cor:prob} below.
In fact, the relations $\bm \cdot \bn = 0$ for some $\bm \in \ZZ^k$ with bounded 1-norm play a role
in~\cite{czerwinski}, and so in some sense we give this role a geometric meaning in terms of lattice widths.

Finally, the combination of Theorem \ref{thm:mainlr}(a) and our computation of the lattice width of lonely runner
zonotopes gives an upper bound for this width when the zonotope is hollow. In Section~\ref{sec:hollowzono}, we give an
accompanying asymptotic lower bound; consequently, the flatness constant for zonotopes grows at least linearly with the
dimension.

% --------------------

\section{Lonely Fourier Analysis}\label{sec:fourier}

The purpose of this section is to prove the following result, which establishes % the linear bound $2k+3$ of 
Theorem \ref{thm:mainlr}(a).
We first need some more terminology. 
We call $t\in\RR$ a \Def{strict lonely time} for $\bn=(n_1,\dots,n_k)$ if $\zdist{t n_j}>\frac{1}{k+1}$ for every $1\le j\le k$. Furthermore, we will call an integer vector $\bm\in\ZZ^k$ a \Def{relation} among the entries of $\bn$ if
$\bm\cdot\bn=0$.
A relation is \Def{harmful} if $m_1+\dots+m_k$ is odd.

\begin{theorem}\label{thm:main}
Let $k\ge2$, and let $n_1,\dots,n_k$ be distinct positive integers.
If $\bn$ has no strict lonely time, then there is a harmful relation $\bm$ such that
\[
\norm{\bm}_2\le \frac{2(k+1)}{\sqrt{k-1}}
\qquad\text{and}\qquad
\norm{\bm}_1\le 2k+3.
\]
In particular, if there is no harmful relation of $1$-norm at most $2k+3$, then $\bn$ has a strict lonely time.
\end{theorem}

The proof intuition is as follows. Beginning with Lemma~\ref{lem:cosine} below, we identify a \emph{window function}
$W_h$ that measures a runner's position on the track (namely the fractional part of $n_jt$), that is positive when a
runner is in the region that leaves the origin (i.e., the zero velocity runner) lonely (i.e., $W_h(n_jt)>0$ when $\{n_jt\}
\in (\frac 1 {k+1}, \frac k {k+1})$). Each runner can be thought of as having its own window function. If the product of the window
functions, denoted $H(t)$, is positive for some time $t$, then the lonely runner conjecture holds in strict form. The
technical mass of this section is focused on an elementary Fourier-analytic study of this function $H(t)$, which helps us understand when strict lonely times may occur.
This strategy is very similar to that employed by Czerwi\'nski in \cite{czerwinski}, however, the technical details of
our analysis diverge, which leads us to our different conclusions. Specifically, we place emphasis on harmful relations,
which are the terms of the series that draw the value of $H(t)$ down; we find this viewpoint quite helpful.

Write $\mathbb T=\RR/\ZZ$, identify it with $[- \frac 1 2, \frac 1 2)$ when convenient, and let $\norm{x}_{\mathbb T}$ denote the distance from $x$ to $0$ on the circle.
For $f\in L^1(\mathbb T)$, let
\[
\widehat f(r) \, := \int_{\mathbb T}f(x)e^{-2\pi i rx}\,dx \, .
\]

The following lemma sets up much of the basic structure required in the remainder of the proof.

\begin{lemma}\label{lem:cosine}
For $0<h< \frac 1 2$, define $\psi_h$ on $[- \frac 1 2, \frac 1 2)$ by
\[
\psi_h(x) \, := \,
\begin{cases}
\sqrt{\frac 2 h}\,\cos \left( \frac{ \pi x } h \right) & \text{ if } |x|\le \frac h 2 \, ,\\
0 & \text{ if } \frac h 2<|x|\le \frac 1 2 \, ,
\end{cases}
\]
and $W_h$ by
\begin{equation}\label{eq:cosine-window}
W_h(x) \, :=\int_{\mathbb T}\psi_h(y) \, \psi_h\!\left(y-x+\tfrac12\right)\,dy.
\end{equation}
Then $W_h$ is continuous and nonnegative, and
\begin{equation}\label{eq:window-support}
W_h(x)>0 \quad\Longleftrightarrow\quad \norm{x-\tfrac12}_{\mathbb T}<h \, .
\end{equation}
Moreover,
\begin{equation}\label{eq:cosine-series}
W_h(x) \, = \, \sum_{r\in\ZZ}(-1)^r a_r \, e^{2\pi i rx}, \quad a_r= \left|\widehat{\psi_h}(r) \right|^2,
\end{equation}
where the coefficients satisfy
\begin{equation}\label{eq:cosine-sums}
\sum_{r\in\ZZ}a_r=1 \qquad \text{and} \qquad \sum_{r\in\ZZ}r^2a_r=\frac1{4h^2} \, .
\end{equation}
The Fourier series in \eqref{eq:cosine-series} converges absolutely and uniformly.
\end{lemma}

\begin{proof}
The integral in \eqref{eq:cosine-window} measures the overlap of two nonnegative cosine bumps ($\psi_h$), one centered
at $0$ and the other at $x- \frac 1 2$. Each bump is supported on an interval of radius $\frac h 2$ and is positive in
the interior of that interval. Their supports therefore overlap on a set of positive length exactly when $\norm{x- \frac 1 2}_{\mathbb T}<h$. This proves \eqref{eq:window-support} and the nonnegativity of $W_h$. Continuity follows from continuity of translations in $L^2(\mathbb T)$ and Cauchy--Schwarz.

The unshifted overlap function
\[
x \, \longmapsto\int_{\mathbb T}\psi_h(y) \, \psi_h(y-x)\,dy
\]
has $r$-th Fourier coefficient $|\widehat{\psi_h}(r)|^2$.
Replacing $x$ by $x- \frac 1 2$ multiplies that coefficient by $e^{-2\pi i r/2}=(-1)^r$.
Hence
\[
\widehat{W_h}(r) \, = \, (-1)^r \left|\widehat{\psi_h}(r) \right|^2= \, (-1)^r a_r \, .
\]
Since $\psi_h$ is real, $a_{-r}=a_r$.

Direct integration gives
\[
\int_{\mathbb T}\psi_h(x)^2\,dx \, = \, \frac2h\int_{-h/2}^{h/2}\cos^2 \left(\frac{ \pi x }{ h } \right) dx \, = \, 1 \, .
\]
Writing $\psi_h'$ for the ordinary derivative inside the support and $0$ outside, direct integration also gives
\[
\int_{\mathbb T}|\psi_h'(x)|^2\,dx \, = \, \frac{2\pi^2}{h^3}\int_{-h/2}^{h/2}\sin^2 \left( \frac{ \pi x } h \right) dx \, = \,
\frac{\pi^2}{h^2} \, .
\]
Parseval applied to the first identity gives $\sum_r a_r=1$.
In particular, the Fourier series in \eqref{eq:cosine-series} converges absolutely and uniformly. Its sum has the same Fourier coefficients as the continuous function $W_h$, so uniqueness of Fourier coefficients gives \eqref{eq:cosine-series}.

Since $\psi_h(\pm \frac h 2)=0$, the piecewise derivative above is the weak derivative of $\psi_h$.
Integration by parts on $[- \frac h 2, \frac h 2]$ yields
\[
\widehat{\psi_h'}(r) \, = \ 2\pi i r \, \widehat{\psi_h}(r) \, .
\]
Parseval applied to $\psi_h'$ finally gives
\[
\sum_{r\in\ZZ}r^2a_r \, = \, \frac1{4\pi^2}\int_{\mathbb T}|\psi_h'(x)|^2\,dx \, = \, \frac1{4h^2} \, . \qedhere
\]
\end{proof}

\begin{proof}[Proof of Theorem~\ref{thm:main}]
Set
\[
h:=\frac{k-1}{2(k+1)} \, , 
\qquad\text{ so that }\qquad
\left(\frac1{k+1}, \ \frac k{k+1}\right) = \left(\frac12-h, \ \frac12+h\right).
\]
Let $W_h$ and $a_r$ be as in Lemma~\ref{lem:cosine}, and define
\[
H(t) \, := \, \prod_{j=1}^k W_h(n_jt) \, .
\]
By \eqref{eq:window-support}, $H(t)>0$ exactly when $t$ is a strict lonely time.
Thus, under the hypothesis of the theorem, $H$ is identically zero. Since the velocities are integers, $H$ is $1$-periodic, so we view it as a function on $\mathbb T$.

For $\bc=(c_1,\dots,c_k)\in\ZZ^k$, put
\[
w(\bc) \, := \, \prod_{j=1}^k a_{c_j} \, .
\]
Using \eqref{eq:cosine-series} in each factor of $H(t)$ gives
\begin{equation}\label{eq:product-expansion}
H(t) \, =\sum_{\bc\in\ZZ^k}(-1)^{c_1+\cdots+c_k}w(\bc)e^{2\pi i(\bc\cdot \bn)t}.
\end{equation}
This expansion is absolutely and uniformly convergent because
\begin{equation}\label{eq:total-weight}
\sum_{\bc\in\ZZ^k}w(\bc) \, = \, \left(\sum_{r\in\ZZ}a_r\right)^k= \, 1 \, .
\end{equation}

For each $q\in\ZZ$, let
\[
E_q \, := \, \left\{\bc\in\ZZ^k: \, \bc\cdot \bn=q,\ \textstyle\sum_j c_j\text{ is even}\right\}
\]
and
\[
O_q \, := \, \left\{\bc\in\ZZ^k: \, \bc\cdot \bn=q,\ \textstyle\sum_j c_j\text{ is odd} \right\} .
\]
Since $H$ is identically zero, its $q$-th Fourier coefficient vanishes.
Equation \eqref{eq:product-expansion} therefore yields
\[
\sum_{\bc\in E_q}w(\bc) \, =\sum_{\bc\in O_q}w(\bc) \, .
\]
Let $s_q$ denote this common value.
Every $\bc\in\ZZ^k$ belongs to exactly one of the sets $E_q$ or $O_q$, so \eqref{eq:total-weight} gives
\begin{equation}\label{eq:sq-sum}
\sum_{q\in\ZZ}s_q \, = \, \frac12 \, .
\end{equation}

By Cauchy--Schwarz and \eqref{eq:cosine-sums}, $\sum_r|r|a_r<\infty$. Since $a_{-r}=a_r$, it follows that
\begin{equation}\label{eq:first-sum}
\sum_{r\in\ZZ}r a_r \, = \, 0 \, .
\end{equation}
Let $\bu=(u_1,\dots,u_k)$ be a unit vector in $\RR^k$.
Expanding the square and using \eqref{eq:cosine-sums} and \eqref{eq:first-sum}, we obtain
\begin{equation}\label{eq:weighted-square}
\sum_{\bc\in\ZZ^k}w(\bc)(\bu\cdot \bc)^2 \, = \, \sum_{j=1}^k u_j^2 \sum_{r\in\ZZ}r^2a_r + 2\sum_{j<\ell}u_ju_\ell\left(\sum_{r\in\ZZ}r a_r\right)^2
 = \, \frac1{4h^2} \, .
\end{equation}
All rearrangements here are justified by the finiteness of $\sum_r|r|a_r$ and $\sum_r r^2a_r$.

Now let $P:\RR^k\rightarrow \RR^k$ be the orthogonal projection onto
\[
\bn^\perp \, := \, \left\{\bz\in\RR^k: \, \bz\cdot \bn=0\right\}.
\]
This space has dimension $k-1$.
Summing \eqref{eq:weighted-square} over an orthonormal basis of $\bn^\perp$ gives
\begin{equation}\label{eq:projected-square}
\sum_{\bc\in\ZZ^k}w(\bc)\norm{P\bc}_2^2 \, = \, \frac{k-1}{4h^2} \, .
\end{equation}
Together with \eqref{eq:sq-sum}, this implies that for some $q$ with $s_q>0$,
\begin{equation}\label{eq:good-q}
\frac1{s_q}\left(\sum_{\bc\in E_q}w(\bc)\norm{P\bc}_2^2 \, +\sum_{\bc\in O_q}w(\bc)\norm{P\bc}_2^2 \right) \, \le \,
\frac{k-1}{2h^2} \, .
\end{equation}
Indeed, if the strict reverse inequality held for every $q$ with $s_q>0$, multiplying by $s_q$ and summing over $q$ would contradict \eqref{eq:projected-square}.

Now we notice that by weighted averaging there is an $\bx\in E_q$ such that
\[
\norm{P\bx}_2^2 \, \le \, \frac1{s_q}\sum_{\bc\in E_q}w(\bc)\norm{P\bc}_2^2 
\]
and $\by\in O_q$ such that
\[
\norm{P\by}_2^2 \, \le \, \frac1{s_q}\sum_{\bc\in O_q}w(\bc)\norm{P\bc}_2^2 \, .
\]
Set $\bm=\bx-\by$.
Since $\bx\cdot \bn=\by\cdot \bn=q$, we have $\bm\cdot \bn=0$. The coordinate sum of $\bx$ is even and that of $\by$ is
odd, so $\sum_j m_j$ is odd (i.e., $\bm$ is a harmful relation). Moreover, $\bm\in \bn^\perp$, whence
\[
\bm \, = \, P(\bx-\by) \, = \, P\bx-P\by 
\]
and, via \eqref{eq:good-q},
\[
\norm{\bm}_2^2 \, \le \, 2\norm{P\bx}_2^2+2\norm{P\by}_2^2 \, \le \, \frac{k-1}{h^2} \, .
\]
Therefore
\begin{equation}\label{eq:l2-bound}
\norm{\bm}_2 \, \le \, \frac{\sqrt{k-1}}h \, = \, \frac{2(k+1)}{\sqrt{k-1}} \, .
\end{equation}

It remains to bound the 1-norm of $\bm$. By Cauchy--Schwarz and \eqref{eq:l2-bound},
\[
\norm{\bm}_1 \, \le \, \sqrt k\,\norm{\bm}_2 \, \le \, 2(k+1)\sqrt{\frac k{k-1}} \, .
\]
Using $\sqrt{1+x}<1+ \frac x 2$ for $x>0$, we obtain
\[
\norm{\bm}_1
\, < \, 2(k+1)\left(1+\frac1{2(k-1)}\right) \, = \, 2k+3+\frac2{k-1}
\, \le \, 2k+5 \, .
\]
For every integer $z$, we have $|z|\equiv z\pmod2$. Hence
\[
\norm{\bm}_1 \, = \, \sum_j|m_j| \, \equiv \, \sum_j m_j\pmod2 \, ,
\]
so $\norm{\bm}_1$ is an odd integer. Since it is strictly less than the odd integer $2k+5$, it is at most $2k+3$.
\end{proof}

% --------------------

\section{Lonely Runner Geometry}\label{sec:geometry}

We now set up notation to capture the shifted lonely runner setup.
For $\bn \in \ZZ^k_{ >0 }$ and $\bs \in \RR^k$, let
\[
  \lon(\bn, \bs) \, := \, \max_{ t \in [0,1] } \, \min_{ 1 \le j \le k } \zdist{ t n_j + s_j } \, ,
\]
which we may interpret as the measure of loneliness of the runner configuration with velocities~$\bn$ and starting
points~$\bs$.

Let $0 < \lambda < \frac 1 2$ and $\gamma \ge 1$.
For given $\bn \in \ZZ^k_{ >0 }$ and $\bs \in \RR^k$, we consider the zonotope (Minkowski sum of line segments)
\begin{align}
  \cZ_{\lambda, \gamma}(\bn, \bs) \, :=& \, \bs + [\lambda, 1 - \lambda]^k + [0,\gamma] \, \bn \nonumber \\
  =& \, \bs + \lambda \, \bone + \sum_{ j=1 }^{ k } [0, 1 - 2 \lambda] \, \be_j + [0,\gamma] \, \bn \, ,
\label{eq:zonostructure}
\end{align}
where $\be_j$ denotes the $j\th$ unit vector and $\bone := \be_1 + \dots + \be_k$,
and its infinite counterpart
\[
  \cZ_{\lambda, \infty}(\bn, \bs) \, := \, \bs + [\lambda, 1 - \lambda]^k + \RR \, \bn \, .
\]
For $\lambda = \frac{ 1 }{ k+1 }$ and $\bs = \bzero$, this is the \emph{lonely runner polyhedron} studied in~\cite{lonelyrunnerpoly}; it is also closely related to the earlier zonotopes constructed
in~\cite{henzemalikiosis}.
The polyhedral connection to the lonely-runner setup is not hard to see.
 
\begin{proposition}\label{prop:equivalencies}
Given $\bn \in \ZZ^k_{ >0 }$, $\bs \in \RR^k$, $0 < \lambda < \frac 1 2$, and $\gamma \ge 1$, the following are equivalent:
\begin{enumerate}[{\sf (a)}]
  \item $\lon(\bn, \bs) \ge \lambda$;
  \item $\cZ_{\lambda, \infty}(\bn, \bs) \cap \ZZ^k \ne \varnothing$;
  \item $\cZ_{\lambda, \gamma}(\bn, \bs) \cap \ZZ^k \ne \varnothing$.
\end{enumerate}
We have equality in (a) if and only if in addition $\cZ_{\lambda, \infty}(\bn, \bs)$ (and, equivalently, $\cZ_{\lambda, \gamma}(\bn, \bs)$) is hollow, i.e., its integer points are all on the boundary.
\end{proposition}

In particular, LRC is equivalent to the existence of $\bn \in \ZZ^k_{ >0 }$ with distinct
entries such that $\cZ_{\frac{ 1 }{ k+1 } , \infty}(\bn, \bzero) \cap \ZZ^k \ne \varnothing$ (equivalently,
$\cZ_{\lambda, \gamma}(\bn, \bzero) \cap \ZZ^k \ne \varnothing$).

\begin{proof}
This is implicitly shown in~\cite{lonelyrunnerpoly}; for the sake of completeness, we give a self-contained proof
here.

To show the equivalence of (a) and (b), we note that, by definition, $\lon(\bn, \bs) \ge \lambda$ if and only if 
\[
  t \, \bn + \bs \, = \, \bm + [\lambda, 1-\lambda]^k
\]
for some $t \in \RR$ and some $\bm \in \ZZ^k$, that is,
\[
  \bs - [\lambda, 1-\lambda]^k + \RR \, \bn 
\]
contains an integer point. But this, in turn, is equivalent to $\cZ_{\lambda, \infty}(\bn, \bs)$ containing an
integer point.

The equivalence of (b) and (c) follows from
\[
  \cZ_{\lambda, \infty}(\bn, \bs) \, = \bigcup_{ j \in \ZZ } \left( \cZ_{\lambda, 1}(\bn, \bs) + j \, \bn \right)
\]
and the fact that the integer-point structure of each $\cZ_{\lambda, 1}(\bn, \bs) + j \, \bn$ looks identical.
\end{proof}

Here is the geometric version of Theorem~\ref{thm:mainlr}(b) in full generality.

\begin{theorem}\label{thm:maingeom}
Let $0 < \lambda < \frac 1 2$ and $\gamma \ge 1$.
If $\cZ_{\lambda, \gamma}(\bn, \bs)$ is hollow then $\bm \cdot \bn = 0$ for some $\bm \in \ZZ^k$ with $0 < \onenorm \bm \le \frac{ \flt(k) }{ (1 - 2 \lambda) }$.
\end{theorem}

Theorem~\ref{thm:mainlr}(b) follows, via Proposition~\ref{prop:equivalencies}, as the special case $\lambda = \frac{ 1 }{ k+1 }$ and $\bs = 0$. 
As in that result, we remark that we could replace $\flt(k)$ in the statement of Theorem~\ref{thm:maingeom} by the 
flatness constant for symmetric polytopes or that for (our special class of) zonotopes.

\begin{proof}
We first note that by Proposition~\ref{prop:equivalencies}, for fixed $\lambda$, $\bn$, and $\bs$, all $\cZ_{\lambda, \gamma}(\bn, \bs)$ are simultaneously hollow (or not) for any $\gamma \ge 1$.

From the zonotope structure~\eqref{eq:zonostructure}, we immediately compute
\begin{align*}
  \width \left( \cZ_{\lambda, \gamma}(\bn, \bs) \right)
  \, &= \min_{ \bm \in \ZZ^k \setminus \bzero } \ \max_{ \ba, \bb \in \cZ_{\lambda, \gamma}(\bn, \bs) } \bm \cdot
(\ba - \bb) \\
     &= \min_{ \bm \in \ZZ^k \setminus \bzero } \left( |m_1| \dots + |m_k| \right) (1 - 2 \lambda) + \gamma \, |\bm \cdot \bn| \, .
\end{align*}
Now assume $\cZ_{\lambda, \gamma}(\bn, \bs)$ is hollow; then its lattice width is bounded by $\flt(k)$ and,
consequently, when computing the above minimum we can ignore $\bm$ with $|\bm \cdot \bn| \ge 1$ if we choose $\gamma
= \flt(k)$. Thus, with this choice of $\gamma$,
\[
  \width \left( \cZ_{\lambda, \gamma}(\bn, \bs) \right)
  \, = \min_{ {\bm \in \ZZ^k \setminus \bzero} \atop { \bm \cdot \bn = 0 } } (1 - 2 \lambda) \onenorm \bm 
  \, \le \, \flt(k) \, . \qedhere
\]
\end{proof}

We repeat our earlier interpretation, this time of Theorem~\ref{thm:maingeom}, that all $\bn \in \ZZ^k$ for which
$\cZ_{\lambda, \gamma}(\bn, \bs) \cap \ZZ^k = \varnothing$ (equivalently, $\lon(\bn, \bs) < \lambda$) live in the
(finite) union of hyperplanes
\[
  \bigcup \left\{ \bm \cdot \bx = 0 : \, \bm \in \ZZ^k , \ 0 < \onenorm \bm \le \frac{ \flt(k) }{ (1 - 2 \lambda) } \right\} .
\]
This immediately implies a shifted generalization of the aforementioned theorem of Czerwi\'nski~\cite{czerwinski}.

\begin{corollary}\label{cor:prob}
Let $0 < \lambda < \frac 1 2$, $\gamma \ge 1$, and $\bs \in \RR^k$.
Construct $\bn \in \ZZ^k$ by choosing its $k$ distinct entries from $\{ 1, 2, \dots, j \}$ uniformly at random.
Then the probability that $\cZ_{\lambda, \gamma}(\bn, \bs) \cap \ZZ^k \ne \varnothing$, or equivalently that $\lon(\bn,
\bs) \ge \lambda$, tends to $1$ as $j \to \infty$.
\end{corollary}

% --------------------

\section{Hollow Zonotopes with Large Lattice Width}\label{sec:hollowzono}

In the previous section, we showed that, for $\gamma = \flt(k)$,
\[
  \width \left( \cZ_{\lambda, \gamma}(\bn, \bzero) \right)
  \, = \min_{ {\bm \in \ZZ^k \setminus \bzero} \atop { \bm \cdot \bn = 0 } } (1 - 2 \lambda) \onenorm \bm \, .
\]
Now let $\lambda = \lon(\bn)$, so that $\cZ_{\lambda, \gamma}(\bn, \bzero)$ is hollow. Theorem~\ref{thm:mainlr}(a)
implies that the lattice width of $\cZ_{\lambda, \gamma}(\bn, \bzero)$ is bounded above by $2k+3$. 
On the other hand, the lattice width of $\cZ_{\lambda, \gamma}(\bn, \bzero)$ for all known tight instances of LRC is
$< 3$, independent of $k$. This motivates the question whether we can say more.

\begin{theorem}\label{thm:maxwidthomega}
$\displaystyle \max_{ \bn \in \ZZ^k_{ >0 } } \width \left( \cZ_{\lon(\bn), \flt(k)}(\bn, \bzero) \right) \in \Theta(k) \, .$
\end{theorem}

\begin{corollary}\label{cor:flatnesszonotopes}
The flatness constant for $k$-dimensional zonotopes is in $\Omega(k)$.
\end{corollary}

\begin{proof}
We claim that if we form $\bn$ by choosing $k$ distinct entries from $\{ 1, 2, \dots, \lceil \alpha^k \rceil \}$ uniformly at random, the two probabilities
\begin{equation}\label{eq:probabilities}
  \prob \left( \min_{ {\bm \in \ZZ^k \setminus \bzero} \atop { \bm \cdot \bn = 0 } } \onenorm \bm \le \beta k \right)
  \qquad \text{ and } \qquad
  \prob \left( \lon(\bn) \ge \delta \right)
\end{equation}
both go to zero as $k$ increases. Here the constants $\alpha > 1$, $\beta > 0$, and $0 < \delta < \frac 1 2$ will be determined later.

This claim implies that the probability of an $\bn$ satisfying $\min_{ {\bm \in \ZZ^k \setminus \bzero} \atop { \bm
\cdot \bn = 0 } } \onenorm \bm \le \beta k$ or $\lon(\bn) \ge \delta$ goes to zero, and thus there exists $\bn$ for which
\[
  \width \left( \cZ_{\lon(\bn), \flt(k)}(\bn, \bzero) \right)
  \, = (1 - 2 \lon(\bn)) \min_{ {\bm \in \ZZ^k \setminus \bzero} \atop { \bm \cdot \bn = 0 } } \onenorm \bm 
  \, > \, (1 - 2 \delta) \, \beta \, k \, ,
\]
and the theorem follows.

Towards the first probability in~\eqref{eq:probabilities}, let
\[
  \diamond(k, r) \, := \, \left\{ \bm \in \ZZ^k : \, \onenorm m \le r \right\} .
\]
Then, for $0 < x < 1$,
\[
  | \diamond(k, r) |
  \, < \, \sum_{ \bm \in \diamond(k, r) } x^{ \onenorm m - r } 
  \, < \, \sum_{ \bm \in \ZZ^k } x^{ \onenorm m - r } 
  \, = \, x^{ -r } \left( \sum_{ m \in \ZZ } x^{ |m| } \right)^k 
  \, = \, x^{ -r } \left( \frac{ 1+x }{ 1-x } \right)^k .
\]
Now assume that $r \le \beta k$ for some $\beta > 0$; then
\[
  | \diamond(k, r) |
  \, < \, \left( \frac{ 1+x }{ x^\beta (1-x) } \right)^k .
\]
For any $\bm \in \ZZ^k \setminus \bzero$, the probability that $\bm \cdot \bn$ is bounded by $\frac{ 1 }{ \lceil \alpha^k \rceil } \le \alpha^{ -  k }$. Thus
\begin{equation}\label{eq:firstprobestimate}
  \prob \left( \min_{ {\bm \in \ZZ^k \setminus \bzero} \atop { \bm \cdot \bn = 0 } } \onenorm \bm \le \beta k \right)
  \, \le \, | \diamond(k, r) | \, \alpha^{ -  k }
  \, < \, \left( \frac{ 1+x }{ \alpha \, x^\beta (1-x) } \right)^k . 
\end{equation}

Regarding the second probability in~\eqref{eq:probabilities},
let $N := \lceil \alpha^k \rceil$ and $[N] := \{ 1, 2, \dots, N \}$.
We first note that, for a given $t \in \RR$,
\[
  S_t \, := \, \left\{ n \in [N] : \, \zdist{ nt } > \frac 1 3 \right\} 
\]
is sum-free, i.e., for $a, b \in S_t$ we have $a+b \notin S_t$,
% just compute this mod 1
and that any sum-free subset of $[N]$ has cardinality $\le \lceil \frac N 2 \rceil$.
% letting s :=  min(S) makes S and s+S disjoint, both living in [s, N+s], which gives 2|S| ≤ N+1

Given a positive integer $M$, let
\[
  T \, := \, [0,1] \cap \tfrac{ 1 }{ MN } \ZZ \, .
\]
Suppose that $\lon(\bn) \ge \delta$, i.e., $\zdist{ t n_j } \ge \delta$ for all $j$, or equivalently
\[
  \zdist{ t n_j - \frac 1 2 } \, \le \, \frac 1 2 - \delta \, .
\]
Choose $s \in T$ with $|s - t| \le \frac{ 1 }{ MN }$, so that
\[
  \left| s n_j - t n_j \right| \, \le \, \frac 1 M \, .
\]
Then, by the triangle inequality,
\[
  \zdist{ s n_j - \frac 1 2 }
  \, \le \, \zdist{ t n_j - \frac 1 2 } + \left| s n_j - t n_j \right|
  \, \le \, \frac 1 M + \frac 1 2 - \delta \, .
\]
Choosing $M = 48$ and $\delta = \frac 3 8$ gives an upper bound of $\frac 1 6$, so that
$n_1, \dots, n_k \in S_s$, and $|S_s| \le \lceil \frac N 2 \rceil$.
Thus
\[
  \prob \left( n_1, \dots, n_k \in S_s \right)
  \, \le \, \left( \frac{ \lceil \frac N 2 \rceil }{ N } \right)^k .
     \le \, \frac{ 1 }{ 2^k } \left( 1 + \frac 1 N \right)^k
     \le \, \frac{ 1 }{ 2^k } \, e^{ \frac k N } . 
\]
On the other hand, we showed above that $\lon(\bn) \ge \frac 3 8$ implies the existence of $s \in T$ such that $n_1,
\dots, n_k \in S_s$, and so
\[
  \prob \left( \lon(\bn) \ge \frac 3 8 \right)
  \, \le \, \frac{ 48 N + 1 }{ 2^k } \, e^{ \frac k N } .
  \, < \, \frac{ 48 \alpha^k + 49 }{ 2^k } \, e^{ k \alpha^{ -k } } .
\]
This goes to 0 as $k \to \infty$ as long as $1 < \alpha \le 2$, and \eqref{eq:firstprobestimate} does likewise as long as
\[
  \frac{ 1+x }{ \alpha \, x^\beta (1-x) } \, < \, 1 \, .
\]
We can, e.g., choose $x = \frac{ 1 }{ 10 }$, $\alpha = 2$, and $\beta = \frac 1 5$.
\end{proof}

% --------------------

\section{Open Problems}\label{sec:open}

Aside from LRC, there are other, arguably more mundane problems that present themselves through the above lines of arguments.

Since the shifted LRC does not hold for $k \ge 5$, one should try to compute $\inf_{ \bn, \bs } \lon(\bn, \bs)$ in small dimensions (for $\bn$ with distinct entries).

A slightly more vague, but we believe interesting, question is how $\lon(\bn, \bs)$ behaves statistically on a given hyperplane $\bm \cdot \bx = 0$, for some $\bm \in \ZZ^k$ with small 1-norm.

There are many special families for $\bn$ know for which LRC is true, for example, lacunary sequences, where
consecutive entries of $\bn$ form a ratio with some prescribed lower bound. Can these known theorems be re-proved from
our results?

Theorem~\ref{thm:maxwidthomega} implies the existence of tight instances of LRC whose corresponding zonotopes have large
lattice width, unlike all known tight instances. Can these new tight instances be constructed explicitly?

Further on the geometric side, it would be nice to find better (and explicit) bounds for the flatness constant of lonely runner zonotopes, even if only in small dimensions.
More generally, Corollary~\ref{cor:flatnesszonotopes} gives a linear asymptotic lower bound for the flatness constant of
zonotopes. We know from~\cite{banaszczyk} the upper bound $O(k \log k)$. What is the true asymptotic value?
 
Finally, and again a bit more vaguely, our results invite the question whether there are any connections to Schmidt's
\emph{subspace theorem}~\cite{schmidtsubspacethm} that points of small height in projective space lie in a finite number of hyperplanes; note that it has a consequence for Diophantine approximation. 

% --------------------

% \newpage
\bibliographystyle{amsplain}
\bibliography{bib}

\setlength{\parskip}{0cm} 

\end{document}